\documentclass[reqno]{amsart}
\newcommand \datum {Dec.\ 27, 2021}
\usepackage{amsmath,amscd}
\usepackage{graphicx}
\usepackage{color}
\usepackage[dvipsnames]{xcolor} 
\usepackage{enumerate}
\numberwithin{equation}{section}
\theoremstyle{plain}
 \newtheorem{theorem}{Theorem}[section]
 \newtheorem{lemma}[theorem]{Lemma}
 \newtheorem{proposition}[theorem]{Proposition}
 \newtheorem{corollary}[theorem]{Corollary}
 \newtheorem{observation}[theorem]{Observation}
\theoremstyle{definition}

 \newtheorem{example}[theorem]{Example}
 \newtheorem{remark}[theorem]{Remark}
\theoremstyle{remark}
 \newtheorem{case}{Case}

\newcommand \Con[1] {\textup{Con}(#1)}
\newcommand \Quo[1] {\textup{Quo}(#1)}
\newcommand \RCon[1] {\textup{RCon}(#1)}

\newcommand \chain [1] {\mathsf C_{#1}}
\newcommand \Nplu {\mathbb N^+}
\newcommand \Nfrom [1] {\mathbb N^{\geq #1}}
\newcommand \Ret[1] {\textup{Ret}(#1)}
\newcommand \restrict [2] {#1\rceil_{\kern -1pt #2}}
\newcommand \sts [1] {\textup{StrS}(#1)}
\newcommand \isc [1] {\textup{ISkC}(#1)}
\newcommand \id [1]  {\textup{id}_{#1}}
\newcommand \abp {\textup{AP}}
\newcommand \abul {A^\bullet}
\newcommand \xstar {X^\star}
\newcommand \eabp {\textup{AP}^+}
\newcommand \set [1]{\{#1\}}
\newcommand \ideal [1] {\mathord{\downarrow}#1}

\newcommand \con   {\textup{con}}
\newcommand \tbf[1]  {\textbf{#1}} 
\newcommand \defiff{\overset{\text{def}}\iff}
\newcommand\red[1]{{\textcolor{red}{#1}}}

\newcommand\blue[1]{{\textcolor{blue}{#1}}}
\newcommand \mgreen [1] {{\color{red!70!green}#1\color{black}}}
\begin{document}
\title[Retracts of rectangular distributive lattices]
{Retracts of rectangular distributive lattices and some related observations}

\author[G.\ Cz\'edli]{G\'abor Cz\'edli}
\email{czedli@math.u-szeged.hu \qquad Address$:$~University of Szeged, Hungary}
\urladdr{http://www.math.u-szeged.hu/~czedli/}

\begin{abstract} 
By a rectangular distributive lattice we mean the direct product of two non-singleton finite chains. We prove that the retracts (ordered by set inclusion and together with the empty set) of  a rectangular distributive lattice $G$ form a lattice, which we denote by $\Ret G$. Also, we describe and count the  retracts  of $G$. Some easy properties of retracts, retractions, and retraction kernels  of (mainly distributive) lattices are observed and several examples are presented, including a 12-element modular lattice $M$ such that $\Ret M$ is not a lattice. 
\end{abstract}

\dedicatory{Dedicated to the memory of my scientific advisor, Andr\'as P. Huhn (1947--1985)}

\thanks{This research was supported by the National Research, Development and Innovation Fund of Hungary under funding scheme K 134851.}



\date{\datum.\hfill{\red{Check the author's website and, later, MathSciNet and Zentralblatt for updates}}}

\[\phantom{mmmmmmmmimmm}
\parbox{9cm}
{\mgreen{Compared to the December 23, 2021 version, \blue{\tbf{only}} some  typos have been corrected; the changes are in \red{\tbf{red}}.\\ }}
\]

\maketitle

\section{Introduction}\label{sect:intro} 
This paper is motivated by \cite{balbes}--\cite{schmid}; in fact, mainly by Jakub\-\'{\i}k\-ov\'a--Studenovsk\'a and P\'ocs \cite{danicap}. The titles of these thirteen
papers speak for themselves.
In spite of these sources, we did not know what the retracts of a non-chain lattice $L$ are and, except for some special lattices discussed in \cite{czg-retrsps}, we did not know any interesting properties of the retracts of $L$. Although we still do not know much, our goal with this paper  is to prove that if $L$ is the direct product of two finite chains, then 
its retracts together with the empty set form a lattice $\Ret L=(\Ret L,\subseteq)$. We describe the structure of $\Ret L$ and determine its size, $|\Ret L|$. Some  easy properties of retracts, retractions, and retraction kernels  of (mainly distributive) lattices are observed and several examples are presented. For example, we give a 12-element modular lattice $M$ such that $\Ret M$ is not a lattice. 
Even if we do not formulate them explicitly, this paper raises some open or ``not-yet-studied'' problems. 

The \emph{retractions} of a lattice $L$ are its idempotent endomorphisms; so they are lattice homomorphism $f\colon L\to L$ such that  $f(f(x))=f(x)$ for all $x\in L$. A sublattice of $L$ is a \emph{retract} of $L$ if it is of the form $f(L):=\set{f(x):x\in L}$ for some retraction $f$ of $L$.

\section{Retraction congruences of algebras with a majority term}
If  $\rho_1$ and $\rho_2$ are relations of algebras $A_1$ and $A_2$, respectively, then $\rho_1\times\rho_2$ is defined to be 
\[
\text{$\set{((x_1,x_2),(y_1,y_2)): (x_1,y_1)\in\rho_1 \text{ and } (x_2,y_2)\in\rho_2}$.}
\]
A \emph{majority term} for a variety $\mathcal V$ of algebras is a ternary term $m(x,y,z)$ 
such that $\mathcal V$ satisfies the identities $m(x,x,y)=m(x,y,x)=m(y,x,x)=x$. The variety of all lattices has majority terms since, say,  $m(x,y,z):=(x\vee y)\wedge (x\vee z)\wedge (y\vee z)$ is such a term. By a \emph{quasiorder} (also known as \emph{preorder}) we mean a reflexive and transitive relation. The set (in fact, the lattice) of congruences 
and that of compatible quasiorders of an algebra $A$ are denoted by $\Con A$ and $\Quo A$, respectively. 
Given a retraction $f\colon L\to L$ of a lattice $L$, $f(L)$ and $\ker(f):=\set{(x,y)\in L^2: f(x)=f(y)}$ are the \emph{retract} and
the \emph{retraction congruence}  associated with $f$. A \emph{retract} and a \emph{retraction congruence} is the retract and the retraction congruence  associated with some retraction $f$.
For lattices $L_1$ and $L_2$, every retraction congruence $\Theta$ of $L_1\times L_2$ is of the form $\Theta_1\times \Theta_2$ with $\Theta_i$ being a congruence of $L_i$ for $i=1,2$ since lattices satisfy the Fraser--Horn property. However, we can say a little bit more. Part (B) of the lemma below is due to Fraser and Horn~\cite[Corollary 1]{fraserhorn}.

\begin{lemma}\label{lemma:FH} If $A_1$ and $A_2$ are algebras in a variety with a majority term, then the following two assertions hold.

\textup{(A)} $\Quo{A_1\times A_2}=\set{\rho_1\times \rho_2: \rho_1\in \Quo {A_1}\text{ and }\rho_2\in\Quo{A_2} }$.

\textup{(B)}(Fraser and Horn~\cite[Corollary 1]{fraserhorn} combined with J\'onsson \cite[Example 1]{jonsson}) 
$\Con{A_1\times A_2}=\set{\Theta_1\times \Theta_2: \Theta_1\in \Con{A_1}\text{ and }\Theta_2\in\Con{A_2} }$.
\end{lemma}

Note that if an algebra $A$ has a majority term, then $\Quo A$ is a distributive lattice by \cite[Corollary 5.2]{czglenkeh}. For a lattice $L$, $\Quo L $ was  described by \cite{czghuhnszabo}, the validity of this description was proved (more economically) again in \cite{czgszabo}, and thoroughly surveyed by Davey~\cite{davey}.

\begin{proof}[Proof of Lemma \ref{lemma:FH}] Observe that in the argument for (A) below, symmetry would  trivially be preserved. 
The ``$\supseteq$'' inclusion in place of the equality ``='' in (A) is trivial. To prove the converse inclusion, let $A:=A_1\times A_2$ and $\rho\in\Quo A$. Define $\rho_1:=\set{(x,y)\in A_1^2: (\exists z\in A_2)\,((x,z),(y,z))\in\rho$}. We claim that 
\begin{equation}
\text{if $(x,y)\in\rho_1$, then for all $t\in A_2$, $((x,t),(y,t))\in\rho$.}
\label{eq:mTKrh}
\end{equation}
To see this, assume that $(x,y)\in\rho_1$ is witnessed by $((x,z),(y,z))\in\rho$. Let $t\in A_2$, and let $m$ be a majority term in the variety containing $A_1$ and $A_2$.  Since $\rho$ is reflexive,
$((x,t),(x,t))\in\rho$ and $((y,t),(y,t))\in\rho$. Since $\rho$ is closed with respect to $m$, we obtain that 
$((x,t),(y,t))
=
( (m(x,x,y),m(z,t,t)), \, ( m(y,x,y), m(z,t,t) ) )
=
m( ((x,z),(y,z)) , ((x,t),(x,t)),  ((y,t),(y,t)))
\in\rho$, proving \eqref{eq:mTKrh}. 

Clearly, $\rho_1$ is reflexive. Its compatibility and transitivity follows trivially by  \eqref{eq:mTKrh}, which allows to use the \emph{same} element $z\in A_2$ witnessing that several pairs belong to $\rho_1$. Hence, $\rho_1\in\Quo{A_1}$. By symmetry, the analogously defined $\rho_2$ belongs to $\Quo{A_2}$. 
Next, we show that, for any $x_1,x_2\in A_1$ and $y_1,y_2\in A_2$,  
\begin{equation}
((x_1,x_2), (y_1,y_2))\in\rho \iff \bigl( (x_1,y_1)\in\rho_1\text{ and }(x_2,y_2)\in\rho_2 \bigr).
\label{eq:sWHbrnD}
\end{equation}
Assume that 
$((x_1,x_2), (y_1,y_2))\in\rho$. By reflexivity,
$((x_1,y_2),(x_1,y_2))\in\rho$ and $((y_1,y_2), (y_1,y_2))\in\rho$. 
Hence,
\begin{multline*}
\red{\Bigl(}(x_1,y_2), (y_1,y_2)\red{\Bigr)}
=
\red{\Bigl(}(m(x_1,x_1,y_1), m(x_2,y_2,y_2))  , (m(y_1,x_1,y_1), m(y_2,y_2,y_2)  )\red{\Bigr)}\cr
=
m\bigl( 
((x_1,x_2),(y_1,y_2)), 
((x_1,y_2),(x_1,y_2)) , 
((y_1,y_2),(y_1,y_2)  \bigr)\in\rho,
\end{multline*}
implying that $(x_1,y_1)\in\rho_1$. We obtain similarly that $(x_2,y_2)\in\rho_2$. Thus, the ``$\Rightarrow$'' part of \eqref{eq:sWHbrnD} holds.  
Conversely, assume that $ (x_1,y_1)\in\rho_1$ and $(x_2,y_2)\in\rho_2$. Using \eqref{eq:mTKrh} and its counterpart for the other component, we obtain  that $ ((x_1,x_2),(y_1,x_2))\in\rho$
and  $((y_1,x_2),(y_1,y_2))\in\rho$. Thus, the transitivity of $\rho$ yields that $((x_1,x_2), (y_1,y_2))$ belongs to $\rho$, completing the argument for \eqref{eq:sWHbrnD}. 

Since  $\rho=\rho_1\times \rho_2$  by \eqref{eq:sWHbrnD}, we have proved part (A) of the lemma. Part (B) follows from the first sentence of the proof.
\end{proof}

For an algebra $A$, let $\RCon A$ denote the \emph{set of retraction congruences} of $A$. That is, $\RCon A$ consists of the kernels of retractions of $A$. 
The goal of this section is to prove the following counterpart of Lemma~\ref{lemma:FH}; lattices belong to its scope.

\begin{proposition}\label{prop:wFwzrvTzn}
If $A_1$ and $A_2$ are algebras in a variety $\mathcal V$ with a majority term and each of $A_1$ and $A_2$ has a singleton subalgebra, then  $\RCon{A_1\times A_2}=\set{\Psi_1\times \Psi_2: \Psi_1\in \RCon {A_1}\text{ and }\Psi_2\in\RCon{A_2} }$.
\end{proposition}

\begin{proof}[Proof of Proposition~\ref{prop:wFwzrvTzn}] Let $m$ be a majority term in $\mathcal V$, and denote $A_1\times A_2$ by $A$. For $i\in\set{1,2}$, let $\set{c_i}$ be a one-element subalgebra of $A_i$. We need the following maps
\begin{align*}
\pi_i&\colon A\to A_i \text{ defined by } (x_1,x_2)\mapsto x_i\text{ for }i\in\set{1,2},\cr
\iota_1&\colon A_1\to A\text{ defined by } x_1 \mapsto (x_1,c_2)
\text{ and } \cr
\iota_2&\colon A_2\to A\text{ defined by }  x_2 \mapsto (c_1,x_2).
\end{align*}
We claim that
\begin{equation}
\parbox{10cm}{if $f\colon A\to A$ is a retraction, then so are
$f_1:=\pi_1\circ f\circ \iota_1\colon A_1\to A_1$ and 
$f_2:=\pi_2\circ f\circ \iota_2\colon A_2\to A_2$, and $\ker f=\ker{f_1}\times\ker{f_2}$.}
\label{pbx:shWrj}
\end{equation}
As a composite of homomorphisms, $f_1$ is a homomorphism, in fact, an endomorphism of $A_1$. For $x\in A_1$, let $(u,v):=f(x,c_2)$ and $(u',v'):=f(u,c_2)$.
Then $u:=f_1(x)$ and $u'=f_1(u)$. 
Let $\Theta=\ker f\in\Con A$.
Since $A$ has the Fraser--Horn property by Lemma~\ref{lemma:FH}, $\Theta=\Theta_1\times \Theta_2$ with $\Theta_1\in\Con{A_1}$ and $\Theta_2\in\Con{A_2}$. 
Using that $f$ is idempotent, we have that 
$f(x,c_2)=(u,v)=f(u,v)$. This means that $((x,c_2),(u,v))\in\Theta$, whereby $(c_2,v)\in\Theta_2$. Since $(u,u)\in \Theta_1$, we have that $((u,c_2), (u,v))\in\Theta_1\times \Theta_2=\Theta$. Thus, 
$(u',v')= f(u,c_2)=f(u,v)=(u,v)$, whence $u'=u$. Hence, $f_1(f_1(x))=u'=u=f_1(x)$, whereby $f_1$ is a retraction of $A_1$. By symmetry, $f_2$ is a retraction of $A_2$. 

To complete the argument for \eqref{pbx:shWrj}, we need to show that
\begin{equation}
\text{for } i\in\set{1,2},\quad \ker {f_i}=\Theta_i.
\label{eq:wWrskFj}
\end{equation}
It suffices to deal with $i=1$.  
Assume that $(x,x')\in \ker{f_1}$. Then $f_1(x)=:u=f_1(x')$,   $f(x,c_2)=(u,v)$ and $f(x',c_2)=(u,v')$ for some $v,v'\in A_2$. 
Since $f$ is idempotent, $f(u,v)=(u,v)$. This  equality and $f(x,c_2)=(u,v)$ give that $( (x,c_2),(u,v))\in\Theta$, whereby $(x,u)\in\Theta_1$. Similarly, $(x',u)\in\Theta_1$. By transitivity and symmetry, we obtain that $(x,x')\in\Theta_1$. Thus, $\ker{f_1}\subseteq \Theta_1$.

Conversely, assume that $(x,x')\in\Theta_1$. Denote $f(x,c_2)$ and $f(x',c_2)$ by $(u,v)$ and $(u',v')$, respectively. Since 
$((x,c_2),(x',c_2))\in\Theta_1\times\Theta_2=\Theta$, we have that 
$(u,v)=(u',v')$. Hence, $f_1(x)=u=u'=f_1(x')$, whence $(x,x')\in\ker{f_1}$. Therefore, $\Theta_1\subseteq \ker{f_1}$, and we have obtained the validity of \eqref{eq:wWrskFj} and that of  \eqref{pbx:shWrj}.

Next, armed with   \eqref{pbx:shWrj},
denote  $\set{\Psi_1\times \Psi_2: \Psi_1\in \RCon {A_1}\text{ and }\Psi_2\in\RCon{A_2}}$ by $H$. If  $\Psi\in \RCon A$, 
then we can pick  a retraction $f\colon A\to A$ with $\ker f=\Psi$, and it follows from  \eqref{pbx:shWrj}
that $\Psi\in H$. Therefore, $\RCon A\subseteq H$.

Conversely, assume that $\Psi=\Psi_1\times \Psi_2\in H$. For $i\in\set{1,2}$, $\Psi_i\in\RCon{A_i}$ allows us to pick a retraction $g_i\colon A_i\to A_i$ with $\ker{g_i}=\Psi_i$. It is obvious that $g_1\times g_2\colon A\to A$, defined by $(x_1,x_2)\mapsto(g_1(x_1),g_2(x_2))$ is a retraction of $A$.  Since $((x_1,x_2),(y_1,y_2))\in \ker {(g_1\times g_2)}
\iff \bigl((x_1,y_1)\in\ker{g_1} \text{ and }(x_2,y_2)\in\ker{g_2}\bigr)
\iff  \bigl((x_1,y_1)\in\Psi_1 \text{ and }(x_2,y_2)\in\Psi_2\bigr)\iff
((x_1,x_2),(y_1,y_2))\in \Psi_1\times \Psi_2=\Psi
$, we have that $\Psi=\ker {(g_1\times g_2)}\in \RCon A$.
Thus, $H\subseteq \RCon A$. Consequently,  $\RCon A =  H$, and the proof of Proposition~\ref{prop:wFwzrvTzn} is complete.
\end{proof}

\begin{example}\label{ex:whWrh} As opposed to retraction congruences, retracts and retractions of direct products of two lattices are not factorizable in general. This is exemplified by the direct square $L$ of the two-element chain $\chain 2$,  its retraction map $f\colon L\to L$ defined by $(x,y)\mapsto (x,x)$, and
the retract $f(L)=\set{(0,0),(1,1)}$.
\end{example}

This example explains that the following observation is only a week statement.

\begin{observation}\label{obs:swggKmz} Let $A_1$ and $A_2$ be algebras. For $i\in\set{1,2}$, let $S_i$ be a retract of $A_i$ and let $f_i\colon A_i\to A_i$ be a retraction. Then  $S_1\times S_2$ is a retract of $A:=A_1\times A_2$, and 
$f_1\times f_2\colon A\to A$ defined by $(x_1,x_2)\mapsto (f_1(x_1), f_2(x_2))$ is a retraction.
\end{observation}

\begin{proof} We can assume that $S_i=f_i(A_i)$. Denote $f_1\times f_2$ by $f$; it is clearly a retraction and $f(A)=S_1\times S_2$.
\end{proof}

The following remark is trivial and does not assume the existence of a majority term, but it will be useful later.

\begin{remark}\label{rem:FXjPT}
If $f\colon A\to A$ is a retraction of an algebra $A$, then 
$f(A)=\set{x\in A: f(x)=x}$.
\end{remark}

\begin{proof}
If $f(x)=x$, then $x=f(x)\in f(A)$ is clear. Conversely, if $x\in f(A)$, then $x$ is of the form $x=f(y)$, whereby $f(x)=f(f(y))=(f\circ f)(y)=f(y)=x$. 
\end{proof}

\section{The main result}
In order to formulate the main result of the paper, some definitions and notations are necessary. Remember that for a lattice $L$,  $\Ret L=(\Ret L,\subseteq)$ stands for the poset (partially ordered set) consisting of the retracts of $L$ and the empty set. It is a bounded poset.
For $n\in\Nplu:=\set{1,2,3,\dots}$, the $n$-element chain is denoted by $\chain n$. Let 
\[\Nfrom 2:=\Nplu\setminus\set 1=\set{2,3,4,\dots}.
\]
By a \emph{grid} we mean the direct product of $G=G_{m,n}:=\chain m\times \chain n$, where $m,n\in\Nfrom 2$. If we want to express its parameters, then we speak of the $m\times n$ grid; it is a distributive lattice with $mn$ elements. According to Gr\"atzer and Knapp~\cite{gratzerknapp1}  and   \cite{gratzerknapp3}, grids are the same as distributive rectangular lattices. 
Although ``grid'' would have been much shorter and quite visual, we use ``distributive rectangular lattice" in the title to make 
it more informative. We may use the notation  $G_{m,n}:=\chain m\times \chain n$ even if $m$ or $n$ is 1, but then $G_{m,n}$ is a chain, not a grid. 
The goal of this section is to prove the following theorem.

\begin{theorem}\label{thmmain} If $m,n\in\Nplu$ and $G=\chain m\times \chain n$, then
$\Ret{G}=\bigl(\Ret{G},\subseteq\bigr)$ is a lattice 
in which the meet operation is the same as forming intersection.
\end{theorem}

Before proving the theorem and formulate a related statement, some preparations are necessary.
Let $G=\chain m\times \chain n$ be a grid. The empty set $\emptyset$ and the sets $A_1\times A_2$ with $\emptyset \neq A_1\subseteq \chain m$ and  $\emptyset \neq A_2\subseteq \chain n$ are called the  \emph{straight subsets} of $G$, while the rest of the subsets are \emph{skew}.  The restriction of a map (= function) $g$ to a set $Y$ is denoted by $\restrict g Y$. 
The \emph{first projection} $G\to\chain m$ is denoted by $\pi_1$ while $\pi_2\colon G\to \chain n$, defined by $(x_1,x_2)\mapsto x_2$ is the \emph{second projection}. A subset $X$ of $G$ is \emph{left injective} if $\restrict {\pi_1}X$ is injective, that is, if $(x_1,x_2)\in X$, $(y_1,y_2)\in X$, and $(x_1,x_2)\neq (y_1,y_2)$ imply that $x_1\neq x_2$. Right injective subsets are analogously defined with the help of $\pi_2$. 
The subset $X\subseteq G$ is called an \emph{injective subset} if it is left injective or right injective. Subsets that are both left and right injective are \emph{doubly injective}. We let
\begin{align*}
\sts G&:=\set{X: X\text{ is a straight subset of }G}\cr
\isc G&:=\set{X: X\text{ is an injective skew chain in }G}.
\end{align*}
Now, as an appendix to Theorem~\ref{thmmain}, we formulate the following statement.

\begin{proposition}\label{propmain}
For $k\in\Nplu$ and the $k$-element finite chain $\chain k$,  $\Ret{\chain k}$ is the $2^k$-element (boolean) powerset lattice consisting of all subsets of $\chain k$. 
For integers $m,n\in \Nfrom 2$ and $G:=G_{m,n}=\chain m\times \chain n$,
the lattice  
$\Ret G$ is the disjoint union of its subsets $\sts{G}$ and $\isc{G}$.
The number of elements of $\sts{G}$ is 
\begin{equation}
|\sts{G}|=1+(2^m-1)(2^n-1)
\label{eq:shRw}
\end{equation}
while that of $\isc{G}$ is
\begin{align}
\phantom{} &\phantom{mmmmmi} |\isc{G}|=\cr
&=
\sum_{s=2}^{\max\set{m,n}}\Biggl(
{m\choose s}\cdot{n+s-1\choose s} +{n\choose s}\cdot{m+s-1\choose s}\cr
&\phantom{=\sum_{s=2}^{\min\set{m,n}}\Biggl(}
- {m\choose s}\cdot{n\choose s}
-n\cdot{m\choose s} - m\cdot{n\choose s} 
\Biggr). \label{eq:wksvmcsmw}
\end{align}
Of course $|\Ret G|= |\sts{G}|+|\isc G|$.
\end{proposition}

\begin{corollary}\label{corolmain} For integers $m,n\geq 2$, the lattice $\Ret {G_{m,n}}$ has a maximal chain consisting of $\max\set{m,n}+2$ elements and a maximal chain consisting of $m+n$ elements. 
\end{corollary}

\begin{figure}[ht]
\centerline
{\includegraphics[width=\textwidth]{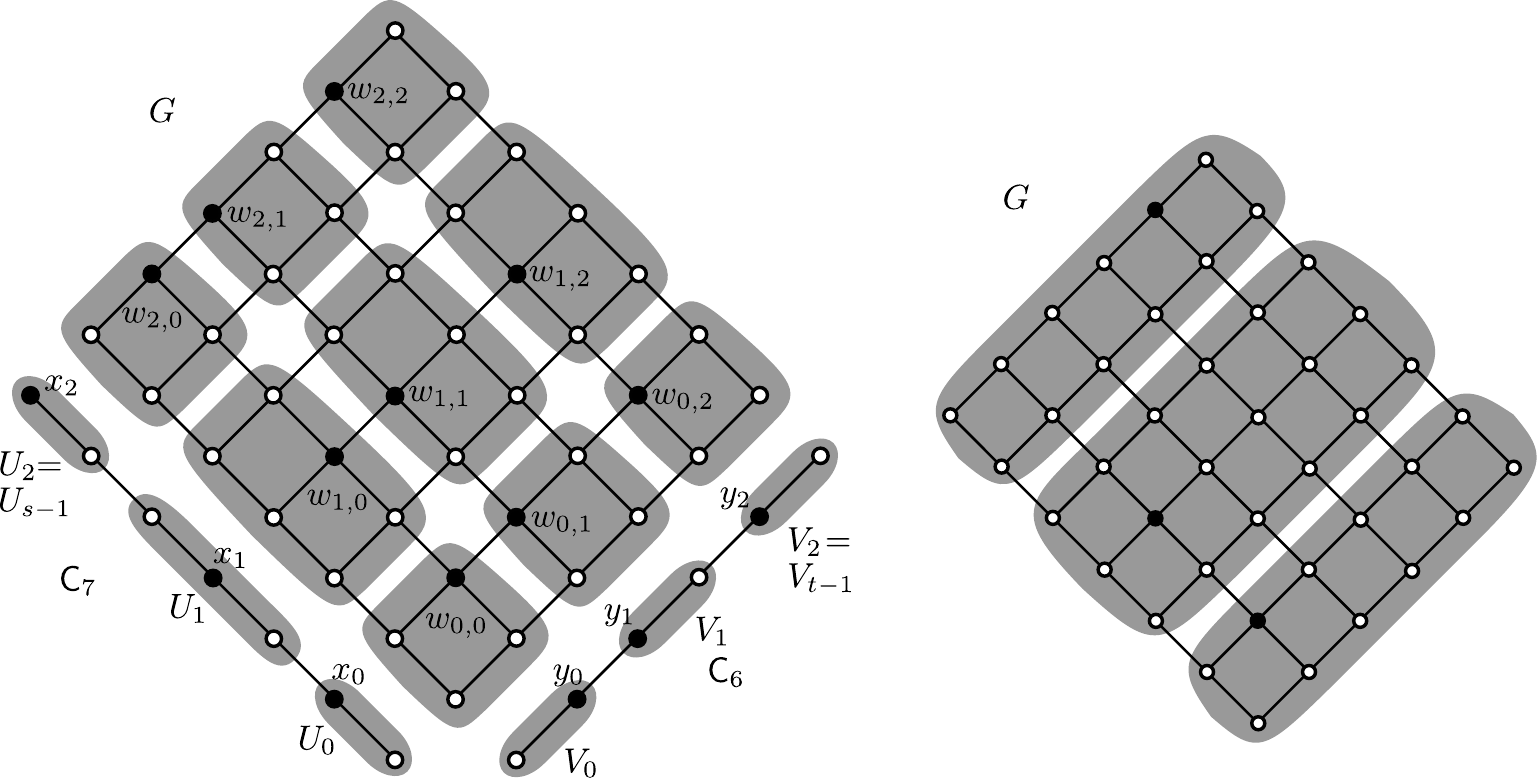}}      
\caption{Illustration to Cases~\ref{case:wWrh} and \ref{case:skwWFrl}}\label{figegy}
\end{figure}

This corollary indicates that the lattice $\Ret{G_{m,n}}$ fails to be distributive in general; in fact, it is not even semimodular or lower semimodular. Note that for a large $n$, $|\isc {G_{n,n}}|$ is much larger than $|\sts {G_{n,n}}|$. For example, computer algebra says that 
\begin{align*}
|\sts {G_{50,50}}|&=1267650600228227149696889520130\approx 1.268\cdot 10^{30},\cr
|\isc {G_{50,50}}|&=17963423287255511675489281668027802959\approx 1.796\cdot 10^{37}\cr
|\Ret {G_{50,50}}|&=17963424554906111903716431364917323089\approx 1.796\cdot 10^{37}.
\end{align*}
Because of space considerations, we only give rounded values for $n=m=1000$:
\begin{align*}
|\sts {G_{1000,1000}}|&  \approx 1.148\,131 \cdot 10^{602},\cr
|\isc {G_{1000,1000}}|&  \approx 7.551\,515  \cdot 10^{763 },\cr
|\Ret{G_{1000,1000}}|& \approx 7.551\,515   \cdot 10^{763 }.\cr
\end{align*} 

The rest of this section is devoted to the proofs of  Theorem~\ref{thmmain}, Proposition~\ref{propmain}, and Corollary~\ref{corolmain}.
The following lemma is trivial; its importance in our proofs justifies that it deserves separate interest. If $\Theta$ is a congruence, then $a/\Theta$ stands for the \emph{$\Theta$-block} $\set{x: (a,x)\in\Theta}$ of $a$.

\begin{observation}\label{obs:drCbngR} Let $A$ be algebra. Then the following hold.

\textup{(A)} A subalgebra $S$ of $A$  is a retract of $A$ if and only if there exists a congruence $\Theta\in\Con A$ such that 
\begin{equation}
\text{for each block $X$ of $\Theta$, we have that  $|X\cap S|=1$.}
\label{eq:stZmsPPswR}
\end{equation}

\textup{(B)} A congruence $\Theta\in\Con A$ is a retraction congruence of $A$ if and only if there exists a subalgebra $S$ of $A$ such that  \eqref{eq:stZmsPPswR} holds.
\end{observation}

\begin{proof} 
To prove (A), assume that $S$ is a retract. Take a retraction $f\colon A\to A$ with $f(A)=S$, and let $\Theta:=\ker f$. For a $\Theta$-block $X$, let $u_X:=f(x_0)$ for some (equivalently, for every) $x_0\in X$. Since $f(u_X)=f(f(x_0))=f(x_0)=u_X$ gives that $(u_X,x_0)\in\Theta$, we have that $u_X\in X$
and $X=\set{y\in A: f(y)=u_X}$. 
By Remark~\ref{rem:FXjPT}, $X\cap S=\set{y\in A: f(y)=u_X\text{ and } f(y)=y}=\set {u_X}$. Hence, \eqref{eq:stZmsPPswR} holds.
Conversely, if \eqref{eq:stZmsPPswR} holds, then $f\colon A\to S$, defined by  the rule $\set{f(x)}=S\cap(x/\Theta)$ is a retraction and $S=f(A)$ is a retract.

To prove (B), let $\Theta\in\Con A$. Assuming that  $\Theta\in\RCon A$, pick a retraction $f\colon A\to A$ with $\ker f=\Theta$, and let $S:=f(A)$. 
Then $S$ is a retract of $A$ and we are in the same situation as after the second sentence of the proof of part (A), whereby 
 \eqref{eq:stZmsPPswR} holds. Conversely, assume that there is a subalgebra $S$ of $A$ such that \eqref{eq:stZmsPPswR} holds.
Then $\Theta$ is the kernel of $f\colon A\to S$, defined by  the rule $\set{f(x)}=S\cap(x/\Theta)$. Since $f$ is a retraction, $\Theta\in\RCon A$, as required. 
\end{proof}

\begin{observation}\label{obs:chain} Let $C$ be a finite chain.
Then every subset of $C$ is a retract, and every congruence of $C$ is a retraction congruence. Also, an equivalence $\Theta$ of $C$ is a congruence if and only if its blocks are intervals.
\end{observation}

\begin{proof} The last sentence is well known. If $\Theta\in \Con C$, then pick an element from each of its blocks; the elements chosen form a sublattice, which is a retract by Observation~\ref{obs:drCbngR}. Now let $S$ be a nonempty subset; it is a sublattice. By an eligible
map we mean a function $g$  from $S$ to the set of intervals of $S$ such that $s\in g(s)$ for all $s\in S$ and $g(s)\cap g(t)=\emptyset$ for any two different $s,g\in S$. Letting $g_1\leq g_2\defiff (\forall s\in S)(g_1(s)\subseteq g_2(s))$, the eligible maps form a finite poset. Now if $g$ is a maximal member of this poset, then $\set{g(s): s\in S}$ is a partition, this partition determines a congruence, and applying Observation~\ref{obs:drCbngR} to this congruence, we conclude that $S$ is a retract.
\end{proof}

\begin{proof}[Proof of Theorem~\ref{thmmain} and Proposition~\ref{propmain}]
In their direct product $G=\chain m\times \chain n$,  we make a distinction  between the two chains even if $m=n$; then the notation $\chain m$ denotes the first chain while $\chain n$ stands for the second one.  Our first task is to prove that 
\begin{equation}
\Ret G=\sts G\cup \isc G.
\label{eq:sJhWmknrc}
\end{equation}
To prove the ``$\supseteq$'' inclusion for \eqref{eq:sJhWmknrc}, assume that  $S\in \sts G\cup \isc G$; we need to show that $S$ is a retract. We can assume that $|S|\geq 2$ since otherwise $S$ is trivially a retract. 
First, let $S\in\sts G$, that is, 
$S=S_1\times S_2$ where $S_1\subseteq \chain m$ and $S_2\subseteq \chain n$ are nonempty subsets. Since $S_1$ and $S_2$ are retracts by (the trivial) Observation~\ref{obs:chain}, $S$ is a retract by Observation~\ref{obs:swggKmz}.  
Second, let $S\in\isc G$. Let, say, $S$ be
a left injective skew chain. Then $\pi_1(S)$ is a retract of $\chain m$ by Observation~\ref{obs:chain}, whereby Observation~\ref{obs:drCbngR} allows us to pick a retraction congruence $\Theta_1\in\Con{\chain m}$ such that for each block $X$ of $\Theta_1$, we have that $|X\cap\pi_1(S)|=1$. Let $\Theta_2=\nabla_{\chain n}$, the largest congruence of $\chain n$, and define $\Theta=\Theta_1\times \Theta_2\in\Con G$. Since $\pi_1$ is injective, each block of $\Theta$ has exactly one element of $S$. Hence $S$ is a retract of $G$ by Observation~\ref{obs:drCbngR}, and we have verified the `$\supseteq$'' inclusion for \eqref{eq:sJhWmknrc}.

To prove the converse inclusion, let $S\in\Ret G\setminus\sts G$; 
we have to show that $S\in\isc G$.  Since $S\notin\sts G$, we know that $|S|\geq 2$. Observation~\ref{obs:drCbngR}(A) allows us to pick a congruence $\Theta\in\Con G$ such that for each $\Theta$-block $X$, we have that $|S\cap X|=1$. 
By the Fraser--Horn property, see Lemma~\ref{lemma:FH}(B), there are $\Theta_1\in\Con{\chain m}$ and $\Theta_2\in\Con{\chain n}$ such that $\Theta=\Theta_1\times \Theta_2$.  Clearly, $\Theta_1$ and $\Theta_2$ are uniquely determined by $\Theta$.  There are two cases.

\begin{case}\label{case:wWrh} We assume that $\Theta_1\neq \nabla_{\chain m}$ and  $\Theta_2\neq \nabla_{\chain n}$; see on the left in Figure~\ref{figegy}, where $m=7$,  $n=6$, and $S$ consists of the black-filled elements. Then 
\begin{equation}
\text{$C_1/\Theta_1$ is a non-singleton chain $\set{U_0\prec U_1\prec\dots\prec U_{s-1}}$,}
\label{eq:swrXRs}
\end{equation}
where $U_0,\dots, U_{s-1}$ are the $\Theta_1$-blocks.  Similarly, $C_2/\Theta_2=\set{V_0\prec V_1\prec\dots\prec V_{t-1}}$ where the $V_j$'s are the $\Theta_2$-blocks. Since $\Theta=\Theta_1\times \Theta_2$, the $\Theta$-blocks are the $U_i\times V_j$'s, $i\in\set{0,1,\dots, s-1}$ and 
$j\in\set{0,1,\dots,t-1}$. Let $w_{i,j}$ denote the unique element of $S\cap (U_i\times V_j)$. We claim that, for $i,i'\in\set{0,1,\dots, s-1}$ and $j,j'\in\set{0,1,\dots,t-1}$,
\begin{equation}
w_{i,j} \wedge w_{i',j'}=w_{\min\set{i,i'},\min\set{j,j'}}
\text{ and }
w_{i,j} \vee w_{i',j'}=w_{\max\set{i,i'},\max\set{j,j'}}.
\label{eq:jmkmpSk}
\end{equation}
To verify \eqref{eq:jmkmpSk}, observe that 
$(U_i\times V_j)\wedge(U_{i'}\times V_{j'})$ (computed in $L/\Theta$) contains $w_{i,j} \wedge w_{i',j'}\in S$ and equals
$U_{\min\set{i,i'}}\times V_{\min\set{j,j'}}$. Since this $\Theta$-block only contains one element from $S$, we obtain the first half of \eqref{eq:jmkmpSk}. Hence, \eqref{eq:jmkmpSk} follows by duality.
Since $\chain m$ and $\chain n$ are  chains, it follows from \eqref{eq:swrXRs}, its counterpart for the $V_j$'s, $w_{i,j}\in U_i\times V_j$, $w_{i',j'}\in U_{i'}\times V_{j'}$, and \eqref{eq:jmkmpSk}  that, for $i,i'\in\set{0,1,\dots,s-1}$ and $j,j'\in\set{0,1,\dots, t-1}$, 
\begin{equation}
\text{if $i\leq i'$ and $j\leq j'$, then $\pi_1(w_{i,j})\leq \pi_1(w_{i',j'})$ and $\pi_2(w_{i,j})\leq \pi_2(w_{i',j'})$.}
\label{eq:pjztmpLg}
\end{equation}

Next, let $x_{s-1}:=\pi_1(w_{s-1,0})$, $y_0:=\pi_2(w_{s-1,0})$,
$x_{0}:=\pi_1(w_{0,t-1})$, and $y_{t-1}:=\pi_2(w_{0,t-1})$. Then
$w_{s-1,0}=(x_{s-1},y_0)$ and  $w_{0,t-1}=(x_0,y_{t-1})$. 
We know from \eqref{eq:pjztmpLg} that $x_0\leq x_{s-1}$ and $y_0\leq y_{t-1}$. These inequalities and \eqref{eq:jmkmpSk} give that $w_{0,0}= w_{s-1,0}\wedge w_{0,t-1}=(x_{s-1},y_0)\wedge(x_0,y_{t-1}) = (x_0,y_0)$. 
Hence, $\pi_1(w_{0,0})=x_0=\pi_1(w_{0,t-1})$ and $\pi_2(w_{0,0})=y_0=\pi_2(w_{s-1,0})$. Thus, 
 \eqref{eq:pjztmpLg} gives that $\pi_2(w_{i,0})= y_0$ and $\pi_1(w_{0,j})=x_0$ for all meaningful $i$ and $j$.
Therefore, letting $x_i=\pi_1(w_{i,0})$ and $y_j:=\pi_2(w_{0,j})$, 
\begin{equation}
w_{i,0}=(x_i,y_0)\,\,\text{ and }\,\, w_{0,j}=(x_0,y_j)
\label{eq:hWfnWrm}
\end{equation} 
for $i\in \set{0,\dots,s-1}$ and $j\in\set{0,\dots,t-1}$.
We know from  \eqref{eq:pjztmpLg} that  $x_0\leq x_1\leq\dots\leq x_{s-1}$ and $y_0\leq y_1\leq\dots\leq y_{t-1}$. 
Since $w_{0,0,}$, $w_{1,0}$, \dots, $w_{s-1,0}$ belong to different $\Theta$-blocks, we have that 
\begin{equation}
\text{$x_0<x_1<\dots<x_{s-1}$ and,  similarly, $y_0<\dots<y_{t-1}$.}
\label{eq:sknWhdnZ}
\end{equation}
Let $X:=\set{x_0,\dots,x_{s-1}}$ and $Y:=\set{y_0,\dots,y_{s-1}}$. Combining \eqref{eq:jmkmpSk}, \eqref{eq:hWfnWrm}, and \eqref{eq:sknWhdnZ}, we obtain that, for all $i\in\set{0,\dots,s-1}$ and $j\in\set{0,\dots, t-1}$, 
\begin{equation*}
w_{i,j}=w_{i,0}\vee w_{0,j}=(x_i,y_0)\vee (x_0,y_j)=(x_i,y_j).
\end{equation*}
Therefore, $S=\set{w_{i,j}: 0\leq i<s\text{ and }0\leq j<t}= \set{(x_i,y_j): 0\leq i<s\text{ and }0\leq j<t}= =X\times Y$. This contradicts the assumption that $S\notin \sts G$, whereby 
Case~\ref{case:wWrh} cannot occur. 
\end{case}

\begin{case}\label{case:skwWFrl} We assume that $\Theta_1 = \nabla_{\chain m}$ or $\Theta_2 = \nabla_{\chain n}$. Both equalities cannot simultaneously hold since otherwise $\Theta=\nabla_G$ would contradict that $|S|>1$. Hence, we can assume that $\Theta_1 \neq \nabla_{\chain m}$ but $\Theta_2 = \nabla_{\chain n}$; see on the right in Figure~\ref{figegy}. 
If $x,y\in S$ such that  $\pi_1(x)=\pi_1(y)$, then $(\pi_1(x),\pi_1(y))\in \Theta_1$ and $(\pi_2(x),\pi_2(y))\in\nabla_{\chain n}=\Theta_2$ gives that $(x,y)\in\Theta_1\times \Theta_2=\Theta$, that is, 
$y\in x/\Theta$, whence $x,y\in S\cap x/\Theta$ yields that $x=y$. Therefore, $\restrict{\pi_1}S$ is injective, that is, $S$ is a left injective subset of $G$.  If we had that $|\pi_2(S)|=1$, then $S=\pi_1(S)\times\pi_2(S)\in\sts G$ would contradict our assumption that $S\in\Ret G\setminus\sts G$. Hence, $|\pi_2(S)|>1$. By way of contradiction, we are going to prove that $S$ is a chain. Suppose to the contrary that this is not so, and pick two incomparable elements $x=(x_1,x_2)$ and $y=(y_1,y_2)$ from $S$. The components of $x$ and $y$ belong to chains, whereby $x\parallel y$ is only possible if either $x_1>y_1$ and $x_2<y_2$, or  $x_1<y_1$ and $x_2>y_2$. By symmetry, we can assume the first alternative, that is, $x_1>y_1$ and $x_2<y_2$. Let $z:=x\vee y=(x_1,y_2)$. Since $S$ is a sublattice, $z\in S$.
Since $\pi_1(x)=x_1=\pi_1(z)$,  we have that $(\pi_1(x),\pi_1(z))\in\Theta_1$. We also have that $(\pi_2(x),\pi_2(z))\in\nabla_{\chain m}=\Theta_2$. Thus, $(x,z)\in\Theta_1\times\Theta_2=\Theta$, which gives that $f(x)=f(z)$. Hence, using that $f$ is order-preserving and $y\leq z$, we have that $y=f(y)\leq f(z)=f(x)=x$, contradicting that $x\parallel y$. Therefore, $S$ is a chain, so it is an injective chain belonging to $\isc G$, as required. 
This completes Case~\ref{case:skwWFrl}, and we have obtained the validity of \eqref{eq:sJhWmknrc}. 
\end{case}

By definition, $\sts G$ and $\isc G$ are clearly disjoint, whence \eqref{eq:sJhWmknrc} imply the second sentence of Proposition~\ref{propmain}. 
The first sentence (about chains) of Proposition~\ref{propmain} is included in Observation~\ref{obs:chain}, which has already been proved.

Next, we turn our attention to the theorem. The rule $(X_1\times Y_1)\cap (X_2\times Y_2)=(X_1\cap X_2)\times(Y_1\cap Y_2)$ shows that $\sts G$ is closed with respect to intersection. So if $X,Y\in\sts G$, then $X\cap Y\in \sts G$, whence \eqref{eq:sJhWmknrc} gives that  $X\cap Y\in \Ret G$. 
Now let $X,Y\in \Ret G$ but, say, $X\notin \sts G$. Then $X$ is an injective skew chain; say, it is left injective. 
Since $X\cap Y$ is a subset of $X$,  we obtain that 
$X\cap Y$ is a left injective chain. If it is not a straight subset, then $X\cap Y\in\isc G\subseteq \Ret G$ by \eqref{eq:sJhWmknrc}). 
If $X\cap Y$ is a straight subset, then $X\cap Y\in \sts G\subseteq \Ret G$ again. Therefore, $\Ret G$ is closed with respect to the binary intersection. By finiteness and since $\Ret G$ has a largest member, $G$, we conclude Theorem~\ref{thmmain}.

Next, $\sts G\cap\isc G=\emptyset$ holds by definition, and \eqref{eq:shRw} is clear. To prove the validity of \eqref{propmain}, note that to obtain an $s$-element left injective chain $X=\set{(x_1,y_1),\dots,(x_s,y_s)}$, we need to select  $(x_1,\dots, x_s)$ and $(y_1,\dots, y_s)$ independently such that $x_1<x_2<\dots<x_s$ and $y_1\leq y_2\leq\dots\leq y_s$.
We can do this in ${m\choose s}\cdot{n+s-1\choose s}$ ways
since $x_i\in \chain m$ and $y_i\in \chain n$ for $i\in\set{1,\dots,s}$. This explains the first summand after the big $\sum$ sign in \eqref{eq:wksvmcsmw}. Note that $m\choose s$ is 0, if $s>m$.
Similarly, the next summand is the number of right injective chains. The sum of the first two summands has to be corrected; first with the number of doubly injective skew chains, then with the number of 
left injective straight chains, and with the number of right injective straight chains; this is where the three subtrahends in \eqref{eq:wksvmcsmw} come from. 
(Since $s\geq 2$, the properties ``doubly injective'', ``left injective and straight'', and ``right injective and straight'' of chains mutually exclude each other.)
Therefore,  \eqref{eq:wksvmcsmw} holds. 
We have proved Theorem~\ref{thmmain} and Proposition~\ref{propmain}.
\end{proof}

\begin{proof}[Proof of Corollary~\ref{corolmain}]
We use the notation $\chain m=\set{0=c_0\prec c_1\prec\dots\prec c_{m-1}=1}$
and $\chain n=\set{0=d_0\prec d_1\prec\dots\prec d_{n-1}=1}$. 
The principal ideals $\ideal c_i$ and $\ideal d_j$ are understood in $\chain m$ and $\chain n$, respectively.
Without loss of generality, we can assume that $m\leq n$. Take the following two chains in $\Ret G=\Ret{\chain m\times \chain n}$:
\begin{align*}
H_1:=\Bigl\{\emptyset,\, &\set{(c_0,d_0)},\,\set{(c_0,d_0),(c_1,d_1)},\,\dots, \, \set{(c_0,d_0),\dots,(c_{m-1},d_{m-1})},\,\cr
& \chain m\times\ideal d_{m-1},\, \chain m\times\ideal d_{m},\,\dots, \chain m\times\ideal d_{n-1}\Bigr \}\quad\text {and} 
\label{eq:Hdgng}\\
H_2:=\Bigl\{&\emptyset,\,\ideal c_0\times\ideal d_0,\,c_1\times\ideal d_0,\,\dots,\,
\ideal c_{m-1}\times\ideal d_0 = \chain m \times \set{d_0},\, \cr
  &\chain m\times \ideal d_1,\dots, \chain m\times \ideal{d_{n-1}}=\chain m\times \chain n\Bigr\}.
\end{align*}
Based on Proposition~\ref{propmain} or \eqref{eq:sJhWmknrc}, it is straightforward to see that both $H_1$ and $H_2$ are maximal chains in $\Ret G$. Since $|H_1|=n+2=\max\set{m,n}+2$ and $|H_2|=m+n$, we obtain Corollary~\ref{corolmain}.
\end{proof}

\section{Some easy facts}
This section collects some easy facts about retracts and related concepts. Some other facts are given in Cz\'edli~\cite{czg-retrsps} and other sections of the present paper.
Recall that an algebra $P$ in a variety $\mathcal V$ is \emph{projective} in $\mathcal V$ if for any algebras $A,B\in \mathcal V$, any  homomorphism $p\colon P\to B$ and any surjective homomorphism $g\colon A\to B$, there is a homomorphism $h\colon P\to A$ such that $p=g\circ h$. In lack of diagonal arrows, this is visualized by the following commutative diagram:
\begin{equation*}
\begin{CD}
A  @>{\text{surjective }{g}}>> B  \\
@A{\exists h}AA   @A p AA   \\
P @<\id P<<  P
\end{CD}
\end{equation*}
The standard category theoretic approach would be to require that $g$ is an epimorphism.
Although there are varieties in which epimorphisms need not be surjective, we go after, say, Freese an Nation~\cite{freesenation} and require $g$ to be surjective rather than just stipulating that $g$ is an epimorphism.  The connection between retracts and projective algebras is well known, say,  from Freese an Nation~\cite{freesenation}. Below, we enlighten another aspect of this connection.

\begin{observation}\label{obs:projective}
If $\Theta$ is a congruence of an algebra $A$ such that $A/\Theta$ is projective in the variety generated by $A$, then $\Theta$ is a retraction congruence.
\end{observation}

\begin{proof} Let $g\colon A\to A/\Theta$ be the natural homomorphism defined by $u\mapsto u/\Theta$; it is surjective.  Let $p$ be the identity map $\id{A/\Theta}\colon A/\Theta\to A/\Theta$. 
Since $A/\Theta$ is projective, there is a homomorphism
$h\colon A/\Theta\to A$ such that $\id{A/\Theta}=g\circ h$.
Now if $X$ is a $\Theta$-block, that is,  $X\in A/\Theta$, 
then $h(X)\in X$ since $X=\id{A/\Theta}(X)=g(h(X))=h(X)/\Theta$. Furthermore, $\set{h(X):X\in A/\Theta}$ is a subalgebra of $A$. Hence, $\Theta\in\RCon A$ by Observation~\ref{obs:drCbngR}(B).
\end{proof}

\begin{figure}[ht]
\centerline
{\includegraphics[scale=0.93]{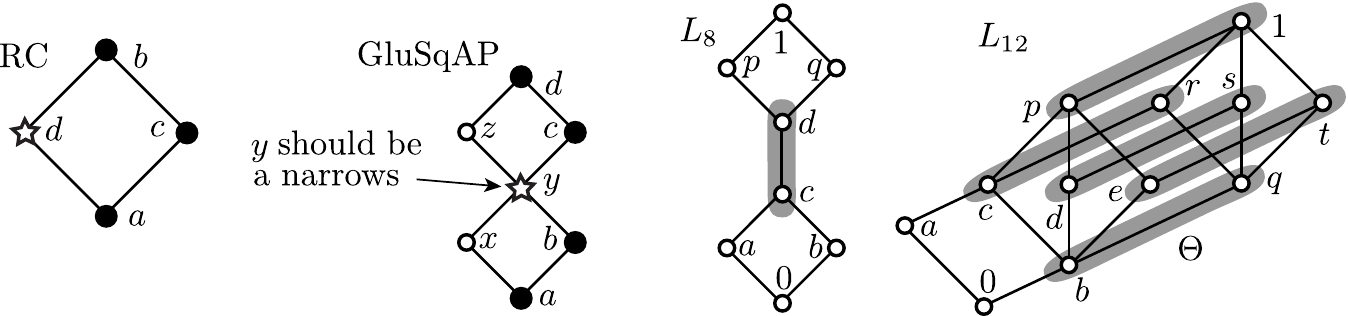}}      
\caption{RC, GluSqAP, $L_8$, and $L_{12}$}\label{figketto}
\end{figure}

If $\abul $ and $\xstar$ are subsets of a lattice $K$, and $\Gamma$ is a property of possible embeddings with domain $K$, then the retracts of a lattice $L$ satisfy the \emph{extended absorption property} $\eabp(K,\abul ,\xstar,\Gamma)$ if for every retract $S$ of $L$ and every embedding $g\colon K\to L$ such that $g$ satisfies $\Gamma$ and  $g(\abul )\subseteq S$, we have that $g(\xstar)\subseteq S$. 
If $\Gamma$ automatically holds for any embedding, then we omit it from the notation and we obtain the \emph{absorption property} $\red{\abp}(K,\abul ,\xstar)$ introduced in Cz\'edli~\cite{czg-retrsps}.
On the left and in the middle of Figure \ref{figketto}, the elements of $\abul$ and $\xstar$ are black-filled  and star-shaped, respectively. The property $\Gamma$, if relevant, is written in the figure.

The simplest absorption property is given on the left of Figure \ref{figketto}; we also denote it by RC. Sublattices satisfying RC are said to be sublattices \emph{closed with respect to taking \tbf{\underline{\underline r}}elative \tbf{\underline{\underline c}}omplements}. 

\begin{observation}\label{obs:RC}
Every retract of a distributive lattice is closed with respect to taking relative complements, that is, the retracts of a distributive lattice satisfy RC.
\end{observation}

\begin{proof}
Let $S$ be a retract of a distributive lattice $L$ and let $f\colon L\to L$ be a retraction with $f(L)=S$. Assume that $a,b,c,d\in L$ form a sublattice isomorphic to the four-element boolean lattice with bottom $a$ and top $b$, and $a,b,c\in S$. Then $f(d)\wedge c=f(d)\wedge f(c)=f(d\wedge c)=f(a)=a$, and we similarly obtain that $f(d)\vee c=b$. Hence, both $d$ and $f(d)$ are complements of $c$ in the interval $[a,b]_L$, which is a distributive lattice.
By the uniqueness of complements in a distributive lattice,  we have that $f(d)=d$, implying that $d\in S$, as required. 
\end{proof}

If $x$ is an element of a lattice $L$, $x\neq 0_L$, $x\neq 1_L$, and $x$ is comparable with every element of $L$, then $x$ is called a \emph{narrows} (of $L$). If we form the glued sum of two squares (i.e., four-element boolean lattices) to obtain a seven-element lattice $K$, then the middle element $y$ of $K$ is a narrows of $K$. However, $y$ need not remain a narrows if we embed $K$ into another lattice. The condition $\Gamma$ on the embedding $g$ we consider in $\textup{GluSqAP}:=\eabp(K,\abul ,\xstar,\Gamma)$ given by Figure~\ref{figketto} is that $g(y)$ should be a narrows.
(The acronym comes from Glued Squares extended Absorption Property.)

\begin{observation}\label{obs:Glu}
The retracts of every lattice satisfy the extended absorption property $\textup{GluSqAP}$.
\end{observation}

\begin{proof} Suppose to the contrary that a retract $S$ of a lattice $L$ fails to satisfy GluSqAP. Then we can assume that $K$, the lattice in the middle left of Figure~\ref{figketto}, is a sublattice of $L$, $y$ is a narrows of $L$, $\set{a,b,c,d}\subseteq S$, but $y\notin S$. Pick a retraction $f\colon L\to L$ that witnesses that $S$ is a retract. By Remark~\ref{rem:FXjPT}, $f(L)=S=\set{x\in L:f(x)=x}$. Since $y\notin S$, $f(y)\neq y$. But $f(y)$ is comparable with $y$ since $y$ is a narrows. Hence, $f(y)>y$ or $f(y)<y$. By duality, we can assume that $y<f(y)$. If we had that $f(x)\leq y$, then $f(y)=f(x\vee b)=f(x)\vee f(b)=f(x)\vee b\leq y$ would contradict the just-assumed $y<f(y)$. Hence, 
$f(x)\not\leq y$. Using again that $y$ is a narrows, we have that $f(x)>y$. Then $a=f(a)=f(x\wedge b)=f(x)\wedge f(b)=f(x)\wedge b\geq y\wedge b=b$ is a contradiction completing the proof.
\end{proof}

\begin{figure}[ht]
\centerline
{\includegraphics[scale=0.93]{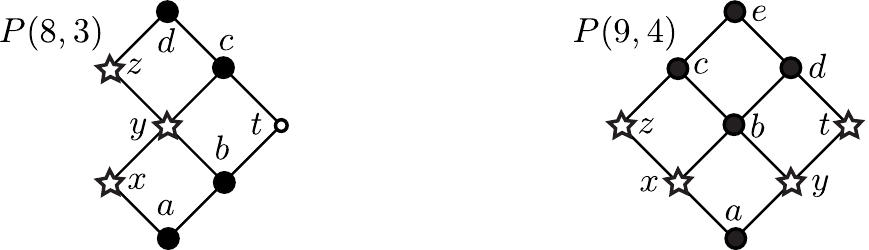}}      
\caption{The absorption properties occurring in Observation~\ref{obs:abspsps}}\label{figharom}
\end{figure}

\begin{observation}\label{obs:abspsps}
The retracts of \emph{planar} distributive lattices satisfy the absorption properties $P(8,3)$ and $P(9,4)$ given in Figure~\ref{figharom} and the dual of  $P(9,4)$.
\end{observation}

\begin{proof} Let $P'(8,3)$ and $P'(9,4)$ be the absorption properties that we obtain from $P(8,3)$ and $P(9,4)$ by
omitting $x$ and $z$ from the $\xstar$ component of $P(8,3)$ and omitting $z$ and $t$ from the $\xstar$ component of $P(9,4)$, respectively.  Visually, to obtain the figure for $P'(8,3)$ from that of $P(8,3)$ we replace the $\star$ by $\circ$ at $x$ and $z$, and analogously for $P'(9,4)$. 
It is proved in Cz\'edli~\cite{czg-retrsps} that the retracts of lattices belonging to a class satisfy $P'(8,3)$ and $P'(9,4)$. 
The class considered there \red{contains} all planar distributive lattices, whereby the retracts of planar distributive lattices satisfy $P'(8,3)$ and $P'(9,4)$. Applying Observation~\ref{obs:RC}, we obtain that they satisfy $P(8,3)$ and $P(9,4)$. This implies Observation~\ref{obs:abspsps} since the concept of a planar distributive lattices is selfdual.
\end{proof}

\section{Examples}

\begin{example}\label{exltkr}
For the lattice $L_{12}$ given in Figure~\ref{figketto}, 
$\Ret{L_{12}}=(\Ret{L_{12}},\subseteq)$ is not a lattice.
\end{example}

\begin{proof} Observe that $[0,a]=\set{0,a}\notin\Ret{L_{12}}$.
Suppose the contrary and take a retraction $f\colon L_{12}\to L_{12}$ such that $f(L_{12})=\set{0,a}$. Then two elements of the "diamond" $[b,p]$ are collapsed by $\ker f$. Since the diamond is a simple lattice, $\ker f$ collapses $b$ and $p$. Hence, 
$a=a\wedge p=f(a\wedge p)=f(a)\wedge f(p)=f(a)\wedge f(b)=f(a\wedge b)=f(0)=0$, which is a contradiction. Thus, $\set{0,a}\notin\Ret{L_{12}}$.  

Let $S_1:=[0,p]$ and $S_2:=[0,a]\cup [q,1]$. 
Both are retracts with the same retraction congruence, the non-singleton \red{blocks} of which are given by the grey ovals. We claim that $\set{S_1,S_2}$ has no greatest lower bound in $\Ret{L_{12}}$. Since any lower bound is a subset of $S_1\cap S_2=\set{0,a}$ but $\set{0,a}\notin\Ret{L_{12}}$, there are at most three lower bounds, $\emptyset$, $\set 0$, and $\set a$. They are retracts, whence there are exactly three lower bounds, $\emptyset$,  $\set 0$ and $\set a$. Since none of these three sets is larger than the other two, $S_1\wedge S_2$ does not exist in $\Ret{L_{12}}$, whereby $\Ret{L_{12}}$ is not a lattice.
\end{proof}

\begin{remark}\label{rem:shtWlsk} 
$\RCon{L_{12}} = \Con{L_{12}}$, and it is the eight-element boolean lattice. 
\end{remark}

\begin{proof} Since the congruence lattice of a finite modular lattice is boolean by Gr\"atzer \cite[Theorem 357]{ggglt},  $\Con{L_{12}}$ is a boolean lattice.
The atoms in $\Con{L_{12}}$ are the principal congruences 
$\con(0,a)$, $\con(0,b)$, and $\con(b,q)$, whereby $|\Con{L_{12}}|=8$ and it is easy to list the congruences of $L_{12}$. For each congruence $\Psi\neq \nabla_{L_{12}}$, there are two easy ways to conclude that $\Psi\in\RCon{L_{12}}$. First, we can easily give a retraction with kernel $\Psi$. 
Second, we can use the criterion given by Balbes~\cite{balbes} to see that $L_{12}/\Psi$ is projective in the variety of distributive lattices, and then $\Psi\in\RCon{L_{12}}$ follows from Observation~\ref{obs:projective}.
\end{proof}

Related to Remark~\ref{rem:shtWlsk}, there is another one.

\begin{remark}\label{rem:szmknSsK}
If $G$ is a the direct product of two finite chains, then 
$\RCon{G} = \Con{G}$, and it is a boolean lattice.
\end{remark}

\begin{proof}
Combine  Proposition~\ref{prop:wFwzrvTzn} and Observation \ref{obs:chain}, and use the well-known fact that the congruence lattice of a finite modular lattice is boolean; see Gr\"atzer~\cite[Theorem 357]{ggglt}.
\end{proof}

Although we do not know  whether $\RCon L$ is always a lattice or when it is a lattice, we can point out that the situation is usually different from what  Remarks~\ref{rem:shtWlsk} and \ref{rem:szmknSsK} describe. Namely, we have the following example. 

\begin{example}\label{ex:nmnDszWhL} For $L_8$ given in Figure~\ref{figketto}, $\RCon{L_8}\neq \Con{L_8}$,  and $\RCon{L_8}$ is a non-distributive lattice.
\end{example}

\begin{proof}Let $\Theta=\con(c,d)$ be the principal congruence indicated in the figure. Except for $\set{c,d}$, its blocks are singletons. Hence, 
$L_8\setminus \set c$ and $L_8\setminus \set d$ are the only candidates for $S$ in Observation~\ref{obs:drCbngR}(B) but none of them is a sublattice. Thus,  $\Theta\notin\RCon{L_8}$,
witnessing that  $\RCon{L_8}\neq \Con{L_8}$. 
Applying  Gr\"atzer \cite[Theorem 357]{ggglt}, it is easy to see that  $\Con{L_8}$ is a boolean lattice consisting of  32 elements. Using Observation~\ref{obs:projective} and  the criterion of  Balbes \cite{balbes}, it is not hard to see that all other congruences are retraction congruences. That is, $\RCon{L_8}=\Con{L_8}\setminus\set{\con(c,d)}$. Hence, $\RCon{L_8}$ is obtained from a finite boolean lattice by omitting an atom.  By the Duality Principle, it suffices to show that
\begin{equation}
\parbox{9cm}{if $d$ is a coatom of a boolean lattice $K$ with $|K|\geq 8$, then the subposet $(K\setminus\set d,\leq)$ is not a distributive lattice.}
\label{eq:pbx:wWh}
\end{equation}
Indeed, the  join-irreducible elements (that is, the elements with exactly one lower cover) are the same in $K$ and 
$K\setminus\set d$, and these elements are antichains in both cases.  If $K\setminus\set d$ was a distributive lattice, then 
the structure theorem of finite distributive lattices, see Gr\"atzer~\cite[Theorem 107]{ggglt}, would give that $K$ and  $K\setminus\set d$  are isomorphic, which is not the case since 
 $K\setminus\set d$ has one element less than $K$.
\end{proof}

\begin{figure}[ht]
\centerline
{\includegraphics[scale=0.93]{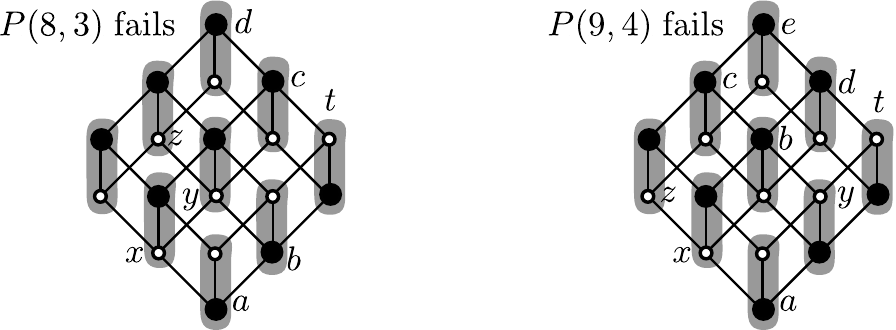}}      
\caption{Where $P(8,3)$ and $P(9,4)$ fail}\label{fignegy}
\end{figure}

\begin{example}\label{ex:whPflS}
The distributive lattice $L=\chain 3\times\chain3\times \chain 2$ has a retract $S$ that satisfies none of $P(8,3)$, nor $P(9,4)$. Moreover, no matter which nonempty subset of $\xstar$ is taken to replace $\xstar$,
$S$ does not satisfy the weaker absorption property we obtain from  $P(8,3)$ or $P(9,4)$ in this way. In  Figure \ref{fignegy}, $L$ is diagrammed twice; $S$ consists of the black-filled elements. 
\end{example}

\begin{proof} In  Figure \ref{fignegy}, a retraction congruence $\Theta$ is given by the grey-filled ovals. Using $\Theta$,  Observation~\ref{obs:drCbngR}(A) shows that $S$ is indeed a retract of $L$.  The embedding is defined by the labeling.
\end{proof}

\end{document}